\documentclass[review]{elsarticle}
\usepackage{geometry}
\usepackage{lineno,hyperref}
\usepackage{dsfont}
\usepackage{tikz}
\usepackage{amsmath,amssymb,amsfonts,amsthm}

\usepackage{algorithm}
\usepackage{bm}
\usepackage{bookmark}
\usepackage{bm}
\modulolinenumbers[5]
\usepackage[all]{xy}

\newtheorem{thm}{Theorem}[section]
\newtheorem{lem}[thm]{Lemma}

\theoremstyle{definition}
\newtheorem{dfn}[thm]{Definition}
\newdefinition{rmk}[thm]{Remark}

\theoremstyle{remark}
\numberwithin{equation}{section}

\journal{Journal of \LaTeX\ Templates}

\makeatletter
\def\ps@pprintTitle{%
   \let\@oddhead\@empty
   \let\@evenhead\@empty
   \def\@oddfoot{\reset@font\hfil\thepage\hfil}
   \let\@evenfoot\@oddfoot
}
\makeatother

\begin{document}

\begin{frontmatter}

\title{A note on unitary Cayley graphs of matrix algebras}
%\tnotetext[mytitlenote]{Fully documented templates are available in the elsarticle package on \href{http://www.\;ctan.\;org/tex-archive/macros/latex/contrib/elsarticle}{CTAN}.\;}
\author[]{Yihan Chen\corref{cor1}}
\ead{2016750518@smail.xtu.edu.cn}
\author[]{Bicheng Zhang\corref{cor1}}
\ead{zhangbicheng@xtu.edu.cn}

\address{School of Mathematics and Computational Science, Xiangtan Univerisity, Xiangtan, Hunan, 411105, PR China}
\cortext[cor1]{Corresponding author}

%\fntext[fn1]{Supported by Hunan Provincial Natural Science Foundation of China (09JJ4002).\;}
%\fntext[fn2]{Supported by Natural Science Foundation of China (11471108).\;}
%\fntext[fn2]{Another author footnote,\; this is a very long
%footnote and it should be a really long footnote.\; But this
%footnote is not yet sufficiently long enough to make two lines
%of footnote text.\;}
%\fntext[fn3]{Yet another author footnote.\;}
%\address{Radarweg 29,\; Amsterdam}
%\fntext[myfootnote]{Since 1880.\;}

%\author[mymainaddress,\;mysecondaryaddress]{Elsevier Inc}
%\ead[url]{www.\;elsevier.\;com}

%\ead{support@elsevier.\;com}
%
%\address[mymainaddress]{1600 John F Kennedy Boulevard,\; Philadelphia}
%\address[mysecondaryaddress]{360 Park Avenue South,\; New York}

\begin{abstract}
Dariush Kiani et al.\cite{kiani2015unitary} claim to have found the unitary Cayley graph $Cay(M_{n}(F),GL_{n}(F))$ of matrix algebras over finite field $F$ is strongly regular only when $n=2$, but they have only considered  two special cases, namely when $n = 2$ and $3$ and have failed to cover the general cases. In this paper, we prove that the unitary Cayley graph of matrix algebras over finite field $F$ is strongly regular if and only if $n=2$.
\end{abstract}
\end{frontmatter}
%\linenumbers
\section{Introduction}
In graph theory, it is of great significance to study the construction and characterization of strongly regular graphs(SRG). The authors of \cite{kiani2015unitary} claimed in their abstract that $n=2$ is a necessary and sufficient condition for $Cay(M_{n}(F),GL_{n}(F))$ to be a SRG, but they only proved $n=2$ is a sufficient condition in the body of their paper. As supplementary, we prove that $Cay(M_{n}(F),GL_{n}(F))$ is SRG if and only if $n=2$.

\section{Preliminaries}
Let $F$ be a finite field, $M_{n}(F)$ be the $n$-dimensional matrix algebra over $F$, $GL_{n}(F)$ be the general linear group.
\begin{dfn}[Unitary Cayley graph]
%We denote $G_{M_{n}(F)}=Cay(M_{n}(F),GL_{n}(F))$, the unitary Cayley graph of $M_{n}(F)$, which is a graph with vertex set $M_{n}(F)$ and edge set $\{\{A,B\}|A-B \in GL_{n}(F)\}$.
Unitary Cayley graph of $M_{n}(F)$ denote by $Cay(M_{n}(F),\\GL_{n}(F))$ or $G_{M_{n}(F)}$ is a graph with vertex set $M_{n}(F)$ and edge set $\{\{A,B\}|A-B \in GL_{n}(F)\}$.
\end{dfn}
\begin{dfn}[Strongly regular graph]
A strongly regular graph with parameter $(n,k,\lambda,\mu)$ is a graph $G$ of order $n$ satisfying the following three conditions: \begin{itemize}
                                                                                                    \item Every vertex adjacents to exactly $k$ vertices.
                                                                                                    \item For any two adjacent vertices $x,y$, there are exactly $\lambda$ vertices adjacent to both $x$ and $y$.
                                                                                                    \item For any two non-adjacent vertices $x,y$, there are exactly $\mu$ vertices adjacent to both $x$ and $y$.
                                                                                                  \end{itemize}
\end{dfn}
\begin{dfn}[Linear derangement]
 $A \in M_{n}(F)$ is a linear derangement if $A \in GL_{n}(F)$ and does not fix any non-zero vector, in other words, $0$ and $1$ are not eigenvalues of $A$.
\end{dfn}
Let $E_{ij}$ be the $n \times n$ matrix with $1$ in the $(i,j)$-position as its only non-zero element, $e_{n}$ be the number of linear derangements in $M_{n}(F)$ and set $e_{0}=0$.
\begin{thm}[See \cite{morrison2006integer}]Let $F$ be a finite field of order $q$, then
$e_{n}=e_{n-1}(q^n-1)q^{n-1}+(-1)^nq^{\frac{n(n-1)}{2}}.$
\end{thm}

\begin{thm}[See \cite{kiani2015unitary}]\label{DKthm}
	Let $F$ be a finite field with order of $q$. Then ${{G}_{{{M}_{2}}}}(F)$ is strongly regular graph with parameters $(q^4, q^4-q^3- q^2 + q, q^4 - 2q^3 - q^2 + 3q, q^4 - 2q^3 + q)$.
\end{thm}

\begin{lem}[See \cite{kiani2015unitary}]\label{en-lemma}
Let $n$ be a positive integer, $F$ be a finite field, then $G=Cay(M_{n}(F),GL_{n}(F))$ is $|GL_{n}(F)|$-regular and for any two adjacent vertices $x,y$, there exists $e_{n}$ vertices adjacent to both of them.
\end{lem}
\begin{lem}\label{det-lemma}
Let $A=(a_{1}, a_{2},\dots, a_{n})$ be a $n\times n$ square matrix with $a_{i}$ the $n$-dimensional column vectors for $i=1,\dots,n$, then $A \in GL_{n}(F)$ and $A+E_{11} \notin GL_{n}(F)$ if and only if $\det(v_{1},a_{2},\dots,a_{n}) \neq 0$ and there exists some $k_{i} \in F$ such that ${a_1} = \sum\limits_{i = 2}^n {{k_i}} {a_i} - {v_1}$ for $v_{1}=(1,0,\dots,0)^T$.
\end{lem}
\begin{proof}
For the necessity, suppose $\det({{v}_{1}},{{a}_{2}},\ldots ,{{a}_{n}})=0$, then $v_{1}$ can be written as a linear combination of $a_{2},\dots,a_{n}$ since $a_{2},\dots,a_{n}$ are linearly independent, then we have
$0 = \det({a_1} + {v_1},{a_2}, \ldots ,{a_n}) = \det({a_1},{a_2}, \ldots ,{a_n})$, a contradiction. Therefore $\det({{v}_{1}},{{a}_{2}},\ldots ,{{a}_{n}})\ne 0$ and $a_{1}+v_{1}$ can be written as a linear combination of $a_{2},\dots,a_{n}$ i.e. there exists some $k_{i} \in F$ such that ${a_1} = \sum\limits_{i = 2}^n {{k_i}} {a_i} - {v_1}$.

For the sufficiency, it is clear that $a_{2},\dots,a_{n}$ are linearly independent and $A+E_{11} \notin GL_{n}(F)$. Suppose $\det(A)=0$, then $a_{1}$ can be written as a linear combination of $a_{2},\dots,a_{n}$, implies
  $\det({v_1},{a_2}, \ldots ,{a_n}) = \det({a_1} + {v_1},{a_2}, \ldots ,{a_n}) = 0$, a contradiction. Therefore $\det(A)\ne 0$, implies $A\in G{{L}_{n}}(F).$
\end{proof}
\section{Results}
Let $F$ be a finite field of order $q$.
\begin{lem}\label{theorem-1}
$|({E_{11}} + G{L_n}(F)) \cap G{L_n}(F)| = ({q^n} - {q^{n - 1}} - 1)\prod\limits_{k = 1}^{n - 1} {({q^n} - {q^k})} $.
\end{lem}
\begin{proof}
For convenience, let $N_{1}$ be the number of matrices $A$ such that $A \in GL_{n}(F)$ and $E_{11}+A \notin GL_{n}(F)$, $N_{2}$ be the number of vector collections $\{a_{2},\dots,a_{n}\}$ such that $\det(v_{1},a_{2},\dots,a_{n}) \neq 0$ and $N_{3}$ be the number of $F$-linear combinations of any $n-1$ linear independent vectors.

By lemma \ref{det-lemma}, we have ${N_1} = {N_2}{N_3}$. It is clear that ${N_3} = {q^{n - 1}}$. For $N_{2}$, to construct such a matrix, for $2\leq k \leq n$, the $k^{th}$ column can be any vector in $F^n$ except for the $q^{k-1}$ linear combinations of the previous $k-1$ columns, hence ${N_2} = ({q^n} - q)({q^n} - {q^2}) \ldots ({q^n} - {q^{n - 1}}) = \prod\limits_{k = 1}^{n - 1} {({q^n} - {q^k})} $.

Therefore
 \begin{align}
  |({E_{11}} + G{L_n}(F)) \cap G{L_n}(F)|  &= |{E_{11}} + G{L_n}(F)| - {N_1}   \nonumber     \\
             &= |G{L_n}(F)| - {N_1} \nonumber\\
             &= \prod\limits_{k = 1}^n {({q^n} - {q^{k - 1}})}  - {q^{n - 1}}\prod\limits_{k = 1}^{n - 1} {({q^n} - {q^k})}\nonumber \\
             &= ({q^n} - 1)\prod\limits_{k = 2}^n {({q^n} - {q^{k - 1}})}  - {q^{n - 1}}\prod\limits_{k = 1}^{n - 1} {({q^n} - {q^k})}\nonumber  \\
             &= ({q^n} - 1)\prod\limits_{k = 1}^{n - 1} {({q^n} - {q^k})}  - {q^{n - 1}}\prod\limits_{k = 1}^{n - 1} {({q^n} - {q^k})} \nonumber \\
             &= ({q^n} - {q^{n - 1}} - 1)\prod\limits_{k = 1}^{n - 1} {({q^n} - {q^k})}
             \nonumber
\end{align}
\end{proof}
\begin{lem}\label{lemma-2}
  $|(diag\{ 1,1,0, \ldots ,0\}  + G{L_n}(F)) \cap G{L_n}(F)|= ({q^{2n}} - {q^{2n - 1}} - {q^{2n - 2}} + {q^{2n - 3}} + {q^{n - 1}} - {q^{n + 1}} + q)\prod\limits_{k = 2}^{n - 1} {({q^n} - {q^k})}$.

\end{lem}
\begin{proof}
For convenience, let $D=diag\{1,1,0,\dots,0\}$, then $|(D  + G{L_n}(F)) \cap G{L_n}(F)|$ is the number of matrices $A$ such that $A \in GL_{n}(F)$ and $A+D \in GL_{n}(F)$. Let $A=(a_{ij}) \in GL_{n}(F)$, then $A+D \in GL_{n}(F)$ if and only if $ I_n+A^{-1}D \in GL_{n}(F).$ Therefore $$|(D  + G{L_n}(F)) \cap G{L_n}(F)| = |\{ A = ({a_{ij}}) \in G{L_n}(F)|I_n + AD \in G{L_n}(F)\} |.$$ Obviously, $I_n + AD \in G{L_n}(F) $ if and only if $ \left( {\begin{array}{*{20}{c}}
{{a_{11}} + 1}&{{a_{12}}}\\
{{a_{21}}}&{{a_{22}} + 1}
\end{array}} \right) \in G{L_2}(F)$, let ${A_1} = \left( {\begin{array}{*{20}{c}}
{{a_{11}}}&{{a_{12}}}\\
{{a_{21}}}&{{a_{22}}}
\end{array}} \right)$, hence $|(D  + G{L_n}(F)) \cap G{L_n}(F)| = |\{ A \in G{L_n}(F)|I_2 + {A_1} \in G{L_2}(F)\} |$.

By Lemma \ref{en-lemma}, the number of matrices $A_{1}$ such that $A_{1} \in GL_{2}(F)$ and $A_{1}+I_2 \in GL_{2}(F)$ is ${e_2} = {q^4} - 2{q^3} - {q^2} = 3q$. Let $M_{1}$ be the number of matrices $B = ({b_{ij}}) \in M_2(F)$ such that
 \begin{center}
  $\left( {\begin{array}{*{20}{c}}
{{b_{11}} + 1}&{{b_{12}}}\\
{{b_{21}}}&{{b_{22}} + 1}
\end{array}} \right)\in GL_{2}(F)$.
\end{center}

Now we determine the value of $M_1$, to construct such a matrix, we can choose any vector in $F^2$ except $(-1,0)^T$ as the first column of $B$, the second column of $B$ can be any vector in $F^2$ except for the $q$ linear combinations of the first column, hence ${M_1} = ({q^2} - 1)({q^2} - q)$.

Consider a matrix $A \in (D  + G{L_n}(F)) \cap G{L_n}(F)$, the rank of $A$'s $2^{nd}$ leading principal submatrix must be $0$, $1$ and $2$. To determine $|(D  + G{L_n}(F)) \cap G{L_n}(F)|$, we distinguish cases pertaining to the rank $r$ of $A$'s $2^{nd}$ leading principal submatrix in the following.
\begin{enumerate}[{\bf {Case} 1.}]
  \item For $r=2$, let $A_{1}$ be a matrix such that $A_{1} \in GL_{2}(F)$ and $A_{1}+I \in GL_{2}(F)$, by similar discussion with $N_2$ in the proof of Lemma \ref{theorem-1}, the number of matrices in $GL_{n}(F)$ have $A_{1}$ as their $2^{nd}$ leading principal submatrix is ${q^{2n - 4}}\prod\limits_{k = 2}^{n - 1} {({q^n} - {q^k})} .$ Therefore, the number of matrices $A \in GL_{n}(F)$ such that $A$ has invertible $2^{nd}$ leading principal submatrix and $A+D \in GL_{n}(F)$ is ${e_2}{q^{2n - 4}}\prod\limits_{k = 2}^{n - 1} {({q^n} - {q^k})}.$
  \item For $r=0$, by similar discussion with $N_2$ in the proof of Lemma \ref{theorem-1}, the number of matrices $A \in GL_{n}(F)$ such that $A$ has $\textbf{0}$ as its $2^{nd}$ leading principal submatrix and $A+D \in GL_{n}(F)$ is $({q^{n - 2}} - 1)({q^{n - 2}} - q)\prod\limits_{k = 2}^{n - 1} {({q^n} - {q^k})} .$
  \item For $r=1$, by the results of cases above, the number of matrices $A \in GL_{n}(F)$ such that the rank of $A$'s $2^{nd}$ leading principal submatrix equals to $1$ and $A+D \in GL_{n}(F)$ is $({M_1} - {e_2} - 1){q^{n - 2}}({q^{n - 2}} - 1)\prod\limits_{k = 2}^{n - 1} {({q^n} - {q^k})}  = [({q^2} - 1)({q^2} - q) - {e_2} - 1]{q^{n - 2}}({q^{n - 2}} - 1)\prod\limits_{k = 2}^{n - 1} {({q^n} - {q^k})} $.
\end{enumerate}

Therefore $|(diag\{ 1,1,0, \ldots ,0\}  + G{L_n}(F)) \cap G{L_n}(F)| = \{ {e_2}{q^{2n - 4}} + ({q^{n - 2}} - 1)({q^{n - 2}} - q) + [({q^2} - 1)({q^2} - q) - {e_2} - 1]{q^{n - 2}}({q^{n - 2}} - 1)\} \prod\limits_{k = 2}^{n - 1} {({q^n} - {q^k})} = ({q^{2n}} - {q^{2n - 1}} - {q^{2n - 2}} + {q^{2n - 3}} + {q^{n - 1}} - {q^{n + 1}} + q)\prod\limits_{k = 2}^{n - 1} {({q^n} - {q^k})}$.
\end{proof}
\begin{thm}\label{bijective}
$A$ and $B$ are two non-adjacent vertices of $G_{M_n}(F)$, then the number of paths of length $2$ between $A$ and $B$ is \[ \left| {\left( {\left( {\begin{array}{*{20}{c}}
{{I_r}}&0\\
0&0
\end{array}} \right) + G{L_n}(F)} \right) \cap G{L_n}(F)} \right|\]
where $r=rank(A-B)$.
\end{thm}
\begin{proof}
Let $N(A)$ be the neighbourhood of $A$, $W$ be the number of paths of length $2$ between $A$ and $B$, then
$ W=|N(A)\cap N(B)|=|(A+G{{L}_{n}}(F))\cap (B+G{{L}_{n}}(F))|.$ Let $T=(A+GL_n(F)) \cap(B+GL_n(F))$, $H=(A-B+G{{L}_{n}}(F))\cap G{{L}_{n}}(F)$. Consider map
\begin{center}
$\begin{array}{l}
\phi :T \to H\\
\;\;M \mapsto M - B
\end{array}$
\end{center}

It is obvious that $\phi$ is injective. For all $K\in H$ we have $K=A-B+X$ with $X\in GL_n(F)$, then $\phi$ is surjective since $\phi(A+X)=K$, therefore $\phi$ is a bijection, implies $W=|H|$.

There exists $P,Q\in G{{L}_{n}}(F)$ such that $P(A-B)Q=
	\begin{pmatrix}
	I_r & 0\\0& 0
	\end{pmatrix}$
 where $r=r(A-B)$.
 Let $S = \left( {\left( {\begin{array}{*{20}{c}}
{{I_r}}&0\\
0&0
\end{array}} \right) + G{L_n}(F)} \right) \cap G{L_n}(F)$, we define the map $\psi$ as follows
\begin{center}
$\begin{array}{l}
\psi :H \to S\\
\;\;\;\;h\;\;\; \mapsto PhQ
\end{array}$
\end{center}
 $\psi$ is obviously bijective, therefore $W = |S| = \left| {\left( {\left( {\begin{array}{*{20}{c}}
{{I_r}}&0\\
0&0
\end{array}} \right) + G{L_n}(F)} \right) \cap G{L_n}(F)} \right|.$
\end{proof}
\begin{thm}\label{notsrg}
For $n>2$, ${{G}_{{{M}_{n}}}}(F)$ is not SRG.
\end{thm}
\begin{proof}
Let $A={{E}_{11}}$, $B=diag\{1,1,0,\cdots ,0\}$, $A$ and $\textbf{0}$ are two non-adjacent vertices in ${{G}_{{{M}_{n}}}}(F)$, neither does $B$, we have

\begin{align}
  |N(A) \cap N(O)|  &= |({E_{11}} + G{L_n}(F)) \cap G{L_n}(F)|\label{e1} \\
   &=({q^n} - {q^{n - 1}} - 1)\prod\limits_{{\rm{k}} = 1}^{n - 1} {({q^n} - {q^k})} \label{se1}
\end{align}
\begin{align}
  |N(B) \cap N(O)|  &= |(diag\{ 1,1,0,...,0\}  + G{L_n}(F)) \cap G{L_n}(F)| \label{e2}\\
   &=({q^{2n}} - {q^{2n - 1}} - {q^{2n - 2}} + {q^{2n - 3}} + {q^{n - 1}} - {q^{n + 1}} + q)\prod\limits_{k = 2}^{n - 1} {({q^n} - {q^k})} \label{se2}
\end{align}
Where (\ref{e1}) and (\ref{e2}) by Theorem \ref{bijective}, (\ref{se1}) by Lemma \ref{theorem-1} and (\ref{se2}) by Lemma \ref{lemma-2}.

For $n>2$, we have $|N(A) \cap N(O)|\neq |N(B) \cap N(O)|$ since ${{q}^{2n-3}}+{{q}^{n-1}}-{{q}^{2n-2}}\ne 0$, hence ${{G}_{{{M}_{n}}}}(F)$ is not SRG.
\end{proof}

\begin{thm}
$Cay(M_{n}(F),GL_{n}(F))$ is SRG if and only if $n=2$.
\end{thm}
\begin{proof}
	By Theorem \ref{DKthm}, Theorem \ref{notsrg} and the obvious fact that ${{G}_{{{M}_{1}}}}(F)$ is not SRG.
\end{proof}

Thus far, we have characterized strongly regular unitary Cayley graphs of matrix algebras over finite field $F$.
\section*{Acknowledgements}
The project is supported partially by Hu Xiang Gao Ceng Ci Ren Cai Ju Jiao Gong Cheng-Chuang Xin Ren Cai (No. 2019RS1057).

\nocite{1}
\nocite{*}
\bibliographystyle{plain}

\bibliography{A_note_on_unitary_Cayley_graphs_of_matrix_algebras}

\end{document}